\documentclass{article}
\usepackage{color}
\usepackage{eurosym}
\usepackage{amsmath}
\usepackage{amsfonts}
\usepackage{verbatim}
\usepackage{graphicx}

\newtheorem{theorem}{Theorem}
\newtheorem{acknowledgement}[theorem]{Acknowledgement}

\newtheorem{definition}[theorem]{Definition}

\newtheorem{lemma}[theorem]{Lemma}
\newtheorem{notation}[theorem]{Notation}

\newtheorem{proposition}[theorem]{Proposition}
\newtheorem{remark}[theorem]{Remark}

\newenvironment{proof}[1][Proof]{\noindent\textbf{#1.} }{\ \rule{0.5em}{0.5em}}

\input{tcilatex}
\begin{document}

\title{Quantitative statistical stability and convergence to equilibrium.\\
An application to maps with indifferent fixed points.}
\author{Stefano Galatolo \thanks{
Dipartimento di Matematica, Universita di Pisa, Via \ Buonarroti 1,Pisa -
Italy. Email: galatolo@dm.unipi.it}}
\maketitle

\begin{abstract}
We show a general relation between fixed point stability of suitably
perturbed transfer operators and convergence to equilibrium (a notion which
is strictly related to decay of correlations). \newline
We apply this relation to deterministic perturbations of a large class of
maps with indifferent fixed points. It turns out that the $L^{1}$ dependence
of the a.c.i.m. on small suitable deterministic changes for these kind of
maps is H\"{o}lder, with an exponent which is explicitly estimated.
\end{abstract}

\section{Introduction}

The statistical stability of dynamical systems is well understood in the
uniformly hyperbolic case (see e.g. \cite{L2} and related references). In
this case, quantitative estimations are available, proving the Lipschitz or
even differentiable dependence of the physical measure under perturbations.
For systems having a non uniformly expanding/hyperbolic behavior, much less
is known. Qualitative results and some quantitative (providing precise
information on the modulus of continuity) ones are known under different
assumptions (see. e.g.\ the survey \cite{ASsu}, the papers \cite{A},\cite{AV}%
,\cite{BBS},\cite{BV},\cite{AT},\cite{SV},\cite{D} and \cite{BB}, for a
recent survey, on differentiable dependence, where also some result beyond
the uniformly expanding case are discussed). While these results do not give
a quantitative information for the general case of deterministic
perturbations of maps with indifferent fixed points, some very recent result
give quantitative information (and even differentiability) for some
deterministic perturbation of certain maps, leading to families of
intermittent maps (\cite{BT},\cite{BS},\cite{Ko}).

In this note we show a general quantitative relation between speed of
convergence to equilibrium and statistical stability. This relation gives
nontrivial information for systems having different speed of convergence to
equilibrium. We show the flexibility of this approach applying it to
perturbations of a class of maps with indifferent fixed points having an
absolutely continuous invariant probability measure. These are systems whose
speed of convergence to equilibrium is bounded by a power law. We show that
this implies the H\"{o}lder continuity of the absolutely continuous
invariant measure under (suitable) deterministic perturbations of the system.

The paper is structured as follows: in Section \ref{1} we show a general
result estimating quantitatively the stability of operator's fixed points
under certain perturbations which is suitable to be applied to estimate the
stability of physical measures. In Section \ref{mann} we apply this result
to a class of maps with \ indifferent fixed points under a large set
deterministic\ perturbations. We remark that our results applies to
perturbations of actual intermittent maps and not only to perturbations of
expanding maps leading to intermittent behavior after perturbation.

\begin{acknowledgement}
{The work was partially supported by EU Marie-Curie IRSES Brazilian-European
partnership in Dynamical Systems (FP7-PEOPLE-2012-IRSES 318999 BREUDS). The
author thanks The Leverhulme Trust for support through Network Grant
IN-2014-021.}
\end{acknowledgement}

\section{Quantitative fixed point stability and convergence to equilibrium. 
\label{1}}

Let us consider a dynamical system $(X,T)$ where $X$ is a metric space and
the space $SM(X)$ of Borel measures with sign on $X$. The dynamics $T$
naturally induces a function $L:SM(X)\rightarrow SM(X)$ which is linear and
is called transfer operator. If $\nu \in SM(X)$ then $L[\nu ]\in SM(X)$ is
defined by%
\begin{equation*}
L[\nu ](B)=\nu (T^{-1}(B))
\end{equation*}%
for every measurable set $B$. If $X$ is a manifold, the measure is
absolutely continuous: $d\nu =f~dm$ (here $m$ represents the Lebesgue
measure) and \thinspace $T$ is nonsingular, the operator induces another
operator $\tilde{L}:L^{1}(m)\rightarrow L^{1}(m)$ acting on the measure
densities ($\tilde{L}f=\frac{d(L(f~m))}{dm}$). By a small abuse of notation
we will still indicate by $L$ this operator.

An invariant measure is a fixed point for the transfer operator. Let us now
see a quantitative stability statement for these fixed points under suitable
perturbations of the operator.

Let us consider a certain system having a transfer operator $L_{0}$ for
which we know the speed of convergence to equilibrium (see (\ref{equil})
below). Consider a "nearby" system $L_{1}$ having suitable properties which
will be specified below: suppose there are two normed vector spaces of
measures with sign $B_{s}\subseteq B_{w}\subseteq SM(X)$ (the strong and
weak space) with norms $||~||_{w}\leq ||~||_{s}$ and suppose the operators $%
L_{0}$ and $L_{1}$ preserve the spaces: $L_{i}(B_{s})\subset B_{s}$ and $%
L_{i}(B_{w})\subset B_{w}$ with $i\in \{0,1\}$.

Let us consider%
\begin{equation*}
V_{s}:=\{f\in B_{s},f(X)=0\}
\end{equation*}%
the space of zero average measures in $B_{s}$. The speed of convergence to
equilibrium of a system will be measured by the speed of contraction to $0$
of this space by the iterations of the transfer operator.

\begin{definition}
\label{d1}Let $\phi (n)$ be a real sequence converging to zero. We say that
the system has \emph{convergence to equilibrium }with respect to norms $%
||~||_{w}$, $||~||_{s}$and speed $\phi $ if $\forall g\in V_{s}$%
\begin{equation}
||L_{0}^{n}(g)||_{w}\leq \phi (n)||g||_{s}.  \label{equil}
\end{equation}
\end{definition}

Suppose $f_{0}$ ,$f_{1}\in B_{s}$ are fixed probability measures of $L_{0}$
and $L_{1}$. The following statement\footnote{%
Similar quantitative stability statements are used in \cite{GN} and \cite{GL}
to support rigorous computation of invariant measures and get quantitative
estimations for the statistical stability of Lorenz like maps.} gives a
quantitative way to estimate the statistical stability of $L_{0}$ under
perturbation if the speed of convergence to equilibrium is known.

\begin{theorem}
\label{gen}Suppose we have estimations on the following aspects of the
system:

\begin{enumerate}
\item (speed of convergence to equilibrium) there is $\phi \in C^{0}(\mathbb{%
R)},~\phi (t)$ decreasing to $0$ as $t\rightarrow \infty $ such that $L_{0}$
has convergence to equilibrium with respect to norms $||~||_{w}$, $||~||_{s}$%
and speed $\phi $.

\item (control on \ the norms of the invariant measures) There is $M\geq 0$
such that%
\begin{equation*}
\max (||f_{1}||_{s},||f_{0}||_{s})\leq M;
\end{equation*}

\item (iterates of the transfer operator are bounded for the weak norm)
there is $\tilde{C}\geq 0$ such that for each $n$, 
\begin{equation*}
||L_{0}^{n}||_{B_{w}\rightarrow B_{w}}\leq \tilde{C}.
\end{equation*}

\item (control on the size of perturbation in the strong-weak norm) Denote 
\begin{equation*}
\underset{||f||_{s}\leq 1}{\sup }||(L_{1}-L_{0})f||_{w}:=\epsilon
\end{equation*}%
consider the decreasing function $\psi $ defined as $\psi (x)=\frac{\phi (x)%
}{x}$, then%
\begin{equation*}
||f_{1}-f_{0}||_{w}\leq (2M+M\tilde{C})\epsilon (\psi ^{-1}(\epsilon )+1).
\end{equation*}
\end{enumerate}
\end{theorem}

\begin{proof}
The proof is a direct computation from the assumptions%
\begin{eqnarray*}
||f_{1}-f_{0}||_{w} &\leq &||L_{1}^{N}f_{1}-L_{0}^{N}f_{0}||_{w} \\
&\leq
&||L_{1}^{N}f_{1}-L_{0}^{N}f_{1}||_{w}+||L_{0}^{N}f_{1}-L_{0}^{N}f_{0}||_{w}
\\
&\leq &||L_{0}^{N}(f_{1}-f_{0})||_{w}+||L_{1}^{N}f_{1}-L_{0}^{N}f_{1}||_{w}.
\end{eqnarray*}%
Since $f_{1}-f_{0}\in V_{s}$ , $||f_{1}-f_{0}||_{s}\leq 2M$,%
\begin{equation*}
||f_{1}-f_{0}||_{w}\leq 2M\phi (n)+||L_{1}^{N}f_{1}-L_{0}^{N}f_{1}||_{w}
\end{equation*}%
but%
\begin{equation*}
(L_{0}^{N}-L_{1}^{N})=\sum_{k=1}^{N}L_{0}^{N-k}(L_{0}-L_{1})L_{1}^{k-1}
\end{equation*}%
hence%
\begin{eqnarray*}
-(L_{1}^{N}-L_{0}^{N})f_{1}
&=&\sum_{k=1}^{N}L_{0}^{N-k}(L_{0}-L_{1})L_{1}^{k-1}f_{1} \\
&=&\sum_{k=1}^{N}L_{0}^{N-k}(L_{0}-L_{1})f_{1}
\end{eqnarray*}%
then%
\begin{equation*}
||f_{1}-f_{0}||_{w}\leq 2M\phi (N)+\epsilon MN\tilde{C}.
\end{equation*}

Now consider the function $\psi $ defined as $\psi (x)=\frac{\phi (x)}{x},$%
chose $N$ such that $\psi ^{-1}(\epsilon )\leq N\leq \psi ^{-1}(\epsilon )+1$%
, in this way $\frac{\phi (N)}{N}\leq \epsilon \leq \frac{\phi (N-1)}{N-1}$,%
\begin{equation*}
||f_{1}-f_{0}||_{w}\leq (2M+M\tilde{C})\epsilon (\psi ^{-1}(\epsilon )+1).
\end{equation*}

\begin{remark}
\label{k1}If $\phi (x)=Cx^{-\alpha }$ then $\psi (x)=Cx^{-\alpha -1}$ and $%
\epsilon (\psi ^{-1}(\epsilon )+1)=C\epsilon (\frac{1}{\sqrt[\alpha +1]{%
\epsilon }}+1)\sim \epsilon ^{1-\frac{1}{\alpha +1}}$ and we have the
estimation for the modulus of continuity%
\begin{equation*}
||f_{1}-f_{0}||_{w}\leq K_{1}\epsilon ^{1-\frac{1}{\alpha +1}}
\end{equation*}%
where the constant $K_{1}$ depends on $M,\tilde{C},C$ and not on the
distance of the operators measured by $\epsilon $.
\end{remark}
\end{proof}

\section{Maps with indifferent fixed points\label{mann}}

In this section we introduce a family of maps with indifferent fixed point;
we then adapt some known results for these kind of maps to the form required
to apply Theorem \ref{gen}.

\begin{definition}
Let $0<\alpha <1$ and let us consider a map $T:[0,1]\rightarrow \lbrack 0,1]$
we say that $T$ is in the set $M(\alpha ,c,C,d)$ if:

\begin{enumerate}
\item $T(0)=0$ and there is a point $\overline{d}\in (d,1)$ such that the
two branches $T|[0,\overline{d}):=T_{1}$ and $T|[\overline{d},1):=T_{2}$ are
continuous and onto.

\item Each branch $T_{i}$ of $T$ is increasing, weakly convex and can be
extended to a $C^{1}$ function;

\item $T^{\prime \prime }$ is continuous everywhere but in the points $0$
and $\overline{d}$; $T^{\prime }>1$ for all $x\in (0,\overline{d})\cup (%
\overline{d},1)$ and $T^{\prime }(0)=1$.

\item There are $\overline{\alpha },\overline{C},\overline{c}\in (0,\infty
), $ $\overline{\alpha }\leq \alpha ,\overline{C}\leq C,\overline{c}\leq c$
such that%
\begin{equation}
|T^{\prime }(x)|=1+\overline{c}x^{\overline{\alpha }}+o(x^{\overline{\alpha }%
})  \label{3.1}
\end{equation}%
\begin{eqnarray*}
|T^{\prime \prime }(x)| &\leq &\overline{C}x^{\overline{\alpha }-1}, \\
|T^{\prime }(x)| &\leq &\overline{C}.
\end{eqnarray*}
\end{enumerate}

We say that $T$ is in the set $M_{0}(\alpha ,c,C,d)$ if moreover $T$ has an
absolutely continuous invariant probability measure with density $h$ such
that.%
\begin{equation*}
\gamma _{0}\leq h(x)
\end{equation*}%
for some $0<\gamma _{0}<1.$ We indicate with $M_{0}$ the set of maps that
are in $M_{0}(\alpha ,c,C,d)$ for some value of the parameters.
\end{definition}

\begin{remark}
\label{stv}We remark that the above conditions are verified for maps of the
following kind: $T(x)=\left\{ 
\begin{array}{c}
x(1+2^{\alpha }x^{\alpha })~\forall x\in \lbrack 0,1/2) \\ 
2x-1~\forall x\in \lbrack 1/2,1]%
\end{array}%
\right. ,$ in particular this map is in $M_{0}$ (see \cite{LSV}).
\end{remark}

\begin{remark}
If a map $T$ is in $M(\alpha ,c,C,d)$, the following condition is then
satisfied: there is a constant $C_{3}\in (0,\infty )$ such that%
\begin{equation}
T(x)\geq x+C_{3}x^{1+\alpha }.  \label{MP}
\end{equation}
\end{remark}

\begin{notation}
Let us denote by $M(\alpha ,c,C,C_{3},d)$ the set of maps in $M(\alpha
,c,C,d)$ satisfying \ref{MP} for a certain $C_{3}$.
\end{notation}

For such a map, we can apply a useful estimation on the asymptotic behavior
of $f$ which can be found in \cite{Mr} (Proposition 1.1, Theorem 1, Equation
3).

\begin{proposition}
Let us consider the transfer operator $L_{T}$ associated to $T\in $ $%
M(\alpha ,c,C,C_{3},d)$ and the following cone of decreasing functions%
\begin{equation*}
C_{A}=\{g\in L^{1}|g\geq
0,g~decreasing,~\int_{0}^{1}g~dm=1,\int_{0}^{x}g~dm\leq Ax^{1-\alpha }\}.
\end{equation*}

Let $A_{\ast }=((1-\alpha )C_{3}d^{2+\alpha })^{-1}$, if $A\geq A_{\ast },$%
then $L(C_{A})\subseteq C_{A}$. Moreover the unique invariant density $f$ of 
$T$ is in $C_{A_{\ast }}$.
\end{proposition}

\begin{remark}
\label{cone}We remark (\cite{Mr}, lemma 2.1) that if $f\in $ $C_{A}$ then $%
f(x)\leq Ax^{-\alpha }$.
\end{remark}

Now we estimate the speed of convergence to equilibrium. The following is
proved in \cite{LSV}

\begin{theorem}
Let us consider $T\in M_{0}(\alpha ,c,C,d),$ if $g\in L^{\infty },f\in
C^{1}, $ $\int f~dx=0$, $0<\gamma <\frac{1}{\alpha }-1$ there is $C_{f}$
depending on $f$ such that%
\begin{equation}
|\int fg\circ T^{n}~dm|\leq C_{f}||g||_{\infty }n^{-\gamma }.  \label{qwq}
\end{equation}
\end{theorem}

Similarly to what is done in \cite{CCS}, Appendix B, for the decay of
correlations, we show that by the uniform boundedness principle it is
possible to see that a statement like (\ref{qwq}) implies a statement on the
convergence to equilibrium which is uniform also in $f$:

\begin{proposition}
\label{6}Let us consider a system satisfying a non uniform convergence to
equilibrium estimation, as in (\ref{qwq}), then there is a constant $C_{2}$
\ such that for each $g\in L^{\infty },f\in C^{1},$ $\int f~dx=0$, $0<\gamma
,$ then%
\begin{equation}
|\int f~g\circ T^{n}~dm|\leq C_{2}||f||_{C^{1}}||g||_{\infty }n^{-\gamma }
\label{qwq1}
\end{equation}%
for some constant $C_{2}.$
\end{proposition}

\begin{proof}
Let $B_{L^{\infty }}$ be the unit ball in $L^{\infty }$. For any $(n,g)\in 
\mathbb{N\times }B_{L^{\infty }}$, let us consider the family of\
non-negative functions defined by%
\begin{equation*}
q_{(n,g)}(f)=n^{-\gamma }|\int fg\circ T^{n}~dm|.
\end{equation*}

For any $f_{1},f_{2},f_{3}\in C^{1}$ it holds%
\begin{equation*}
q_{(n,g)}(f_{1}+f_{2})\leq q_{(n,g)}(f_{1})+q_{(n,g)}(f_{2})
\end{equation*}%
and%
\begin{equation*}
q_{(n,g)}(-f_{3})=q_{(n,g)}(f_{3}).
\end{equation*}%
Moreover it follows from \ref{qwq} that for each $f\in C^{1}$%
\begin{equation*}
\sup_{\mathbb{N\times }B_{L^{\infty }}}q_{(n,g)}(f)\leq C_{f}<\infty .
\end{equation*}%
By the uniform boundedness principle then (see \cite{K}, Theorem 1.29, page
139 ) it follows%
\begin{equation*}
\sup_{\mathbb{N\times }B_{L^{\infty }},||f||_{C_{1}}\leq
1}q_{(n,g)}(f):=C_{2}<\infty .
\end{equation*}%
Which implies \ref{qwq1}
\end{proof}

\begin{remark}
The above proof generalizes to any speed of convergence to equilibrium.
Moreover, since for any Lipschitz funcion $f$ there is $\hat{f}\in C^{1}$
such that $||\hat{f}||_{C^{1}}\leq ||f||_{Lip}$ (where $||~||_{Lip}$ denotes
the Lipschitz norm) and $||f-\hat{f}||_{1}\leq \epsilon $ ($g\in L^{\infty
}, $ thus approximating $f$ in $L^{1}$ is enough to approximate the integral
of the product) then the above theorem also hold for Lipschitz functions $f$
with the Lipschitz norm:%
\begin{equation*}
|\int f~g\circ T^{n}~dm|\leq C_{2}||f||_{Lip}||g||_{\infty }n^{-\gamma }.
\end{equation*}
\end{remark}

\subsection{Application of Theorem \protect\ref{gen} to maps with
indifferent fixed points}

To apply Theorem \ref{gen} we need some technical work to verify the
required assumptions and provide the required estimations. This will be done
in this section and in the following ones. Let us consider the following
norm:%
\begin{equation*}
||f||_{\alpha }=\max (\limfunc{esssup}_{x\in (0,1]}|x^{\alpha }f(x)|,%
\limfunc{esssup}_{x\in (0,1]}|x^{\alpha +1}f^{\prime }(x)|).
\end{equation*}%
We show now how Theorem \ref{gen} can be applied to our class of maps,
considering $||~||_{\alpha }$ and $||~||_{1}$ as strong and weak norms.

The first step is to extend the result of Proposition \ref{6} from Lipschitz
observables \ to a class including functions with finite $||~||_{\alpha }$
norm.

\subsubsection{Convergence to equilibrium for the $||~||_{\protect\alpha }$
norm (item 1 of Theorem \protect\ref{gen})}

In the following proposition we show how a relation like Proposition \ref{6}
implies an estimation of the convergence to equilibrium, as Definition \ref%
{d1} with respect to $||~||_{1}$ and the strong norm $||~||_{\alpha }$

\begin{proposition}
If for a nonsingular function $T$ and some $\gamma >0,$ $C_{2}\geq 0,$it
holds that for each$~g\in L^{\infty }$ and $f\in Lip[0,1],$ with $\int
f~dx=0 $, 
\begin{equation}
|\int f~g\circ T^{n}~dm|\leq C_{2}||f||_{Lip}||g||_{\infty }n^{-\gamma };
\label{7}
\end{equation}%
then there is $C_{4}\geq 0$ such that%
\begin{equation*}
||L^{n}f||_{1}\leq C_{4}n^{-\frac{\gamma }{2}(1-\alpha )}||f||_{\alpha }
\end{equation*}%
where $L$ is the transfer operator associated to $T$.
\end{proposition}

\begin{proof}
Suppose $||g||_{\infty }\leq 1$. There is no loss of generality in supposing 
$f(0)\geq 0$. Let $q\in (0,1]$ small enough such that $f|_{(0,q]}\geq 0$ $,$%
\begin{eqnarray*}
|\int f~g\circ T^{n}~dm| &\leq &|\int_{0}^{q}f~g\circ
T^{n}~dm+\int_{q}^{1}f~g\circ T^{n}~dm| \\
&\leq &||f1_{[0,q]}||_{1}+\frac{1}{2}q^{-\alpha +1}||f||_{\alpha
}+|\int_{0}^{1}f_{q}~g\circ T^{n}~dm|
\end{eqnarray*}%
where $f_{q}(x)=\left\{ 
\begin{array}{c}
q^{-1}f(q)x~\text{for\ }x\leq q \\ 
f(x)~\text{for\ }x\geq q%
\end{array}%
\right. $. Since we have convergence to equilibrium for functions of zero
average we condider $\hat{f}_{q}(x)=f_{q}(x)-\int f_{q}(x)~dm$ . Since $\int
f~dm=0$ we have that $|\int f_{q}(x)~dm|\leq ||f1_{[0,q]}||_{1}+\frac{1}{2}%
q^{-\alpha +1}||f||_{\alpha }$.

Hence 
\begin{eqnarray*}
||L^{n}f||_{1} &\leq &\sup_{||g||_{\infty }\leq 1}|\int f~g\circ T^{n}~dm| \\
&\leq &2||f1_{[0,q]}||_{1}+q^{-\alpha +1}||f||_{\alpha }+||L^{n}\hat{f}%
_{q}||_{1} \\
&\leq &||f||_{\alpha }2\frac{2-\alpha }{1-\alpha }q^{-\alpha
+1}+||L^{n}f_{q}||_{1}.
\end{eqnarray*}%
Since $f_{q}(x)$ is Lipschitz and $||\hat{f}_{q}||_{lip}\leq ||f||_{\alpha
}q^{-\alpha -1}$, by \ref{7}%
\begin{eqnarray*}
||L^{n}\hat{f}_{q}||_{1} &\leq &C_{2}||\hat{f}_{q}||_{Lip}n^{-\gamma } \\
&\leq &C_{2}||f||_{\alpha }q^{-\alpha -1}n^{-\gamma }.
\end{eqnarray*}%
and then 
\begin{eqnarray*}
||L^{n}f||_{1} &\leq &||f||_{\alpha }2\frac{2-\alpha }{1-\alpha }q^{-\alpha
+1}+C_{2}||f||_{\alpha }q^{-\alpha -1}n^{-\gamma } \\
&\leq &||f||_{\alpha }(2\frac{2-\alpha }{1-\alpha }q^{-\alpha
+1}+C_{2}q^{-\alpha -1}n^{-\gamma })
\end{eqnarray*}

choosing $q=\sqrt{n^{-\gamma }}$, then%
\begin{equation*}
||L^{n}f||_{1}\leq C_{4~}||f||_{\alpha }n^{-\frac{\gamma }{2}(1-\alpha )}
\end{equation*}
\end{proof}

\subsubsection{Applying Theorem \protect\ref{gen}}

We can now apply Theorem \ref{gen}, for suitable perturbations, obtaining
that the physical invariant measure varies continuously, with a Holder
continuity modulus. In this we will consider the norm $||~||_{\alpha }$ as a
strong norm and the $L^{1}$ norm as a weak one.

We remark that the transfer operators are $L^{1}$ contractions then the
constant at item 3 of Theorem \ref{gen} satisfies $\tilde{C}=1$. By Theorem %
\ref{gen} it follows:

\begin{proposition}
\label{pto}Let us suppose that $T_{0}\in M_{0}(\alpha ,c,C,d)$ at beginning
of Section \ref{mann}, let $L_{0}$ be its transfer operator and $f_{0}$ its
absolutely continuous invariant measure. Let $T_{1}\in M(\alpha ,c,C,d)$,
let $L_{1}$ be its transfer operator and $f_{1}$ its absolutely continuous
invariant measure. If $T_{0}$ and $T_{1}$ are such that%
\begin{equation}
\max (||f_{1}||_{\alpha },||f_{0}||_{\alpha })\leq M;  \label{12112}
\end{equation}

and $L_{1}$ is a small perturbation of $L_{0}$, in the following sense: 
\begin{equation}
\underset{||f||_{\alpha }\leq 1}{\sup }||(L_{1}-L_{0})f||_{1}\leq \epsilon
\label{1221}
\end{equation}

then for each $0<\gamma <\frac{1}{\alpha }-1$ there is $K_{2}$, depending on 
$M,\tilde{C}=1$ (see Remark \ref{k1}) and $T_{0}$, but not on $\epsilon $
such that 
\begin{equation*}
||f_{1}-f_{0}||_{1}\leq K_{2}\epsilon ^{1-\frac{1}{\frac{\gamma }{2}%
(1-\alpha )+1}}.
\end{equation*}
\end{proposition}

We now show how assumptions (\ref{12112}) and (\ref{1221}) can be naturally
supposed for deterministic perturbations in our class of maps, and we get an
explicit estimation of $M$.

\subsubsection{Bound on the strong norm of the invariant measures (item 2 of
Theorem \protect\ref{gen})}

In this section we show how it is possible to have an estimation of the $%
||~||_{\alpha }$ norm of the absolutely continuous invariant measures of
maps in the set $M(\alpha ,c,C,d)$.

To estimate the $||~||_{\alpha }$ norm of this invariant density, we need to
have an estimation on its derivative. In \cite{LSV} the following is proved:

\begin{lemma}
\label{Mdue}Let $T$ $\in $ $M(\alpha ,c,C,d)$. and $h$ is its a.c.i.m; let $%
c_{T}=||T^{\prime }||_{\infty }$ consider $K_{T},a_{T},b_{T}\in \mathbb{R}$
such that:%
\begin{equation}
K_{T}=\sup_{x}(x^{\alpha -1}T(x));  \label{M2}
\end{equation}%
\begin{equation}
a_{T}>\sup_{x}(\frac{4CK_{T}}{|D_{x}T|^{2}});  \label{M3}
\end{equation}%
\begin{equation}
b_{T}>a_{T}\sup_{x}(\frac{2c_{T}T(x)-x|D_{x}T|}{(|D_{x}T|-2c_{T})T(x)x}).
\label{M4}
\end{equation}%
Then the cone%
\begin{equation*}
C_{0}=\{f\in C^{1}(]0,1])~|~0\leq f(x)\leq
2h(x)\int_{0}^{1}f~dx~;~|f^{\prime }(x)|\leq \frac{a_{T}+b_{T}x}{x}f(x)\}
\end{equation*}%
is invariant for the transfer operator$.$
\end{lemma}

\begin{proof}
In \cite{LSV} (Lemma 5.1) is proved that, if $f\in C_{0}$ then%
\begin{equation*}
|(Lf)^{\prime }(x)|\leq \frac{a-bx}{x}Lf(x)\underset{y\in \lbrack 0,1]}{\sup 
}[\frac{2C}{|D_{y}T|^{2}}\frac{y^{\alpha -1}T(y)}{a+bT(y)}+\frac{T(y)}{%
y|D_{y}T|}\frac{a+by}{a+bT(y)}].
\end{equation*}%
The result follows by remarking that for $a,\frac{b}{a}$ big enough%
\begin{equation*}
\underset{y\in \lbrack 0,1]}{\sup }[\frac{2C}{|D_{y}T|^{2}}\frac{y^{\alpha
-1}T(y)}{a+bT(y)}+\frac{T(y)}{y|D_{y}T|}\frac{a+by}{a+bT(y)}]<1.
\end{equation*}

We need to have an explicit estimation of $a,b$ in function of $T$ . Since $%
T(x)=x+\frac{c}{\alpha +1}x^{\alpha +1}+o(x^{\alpha +1})$ then $y^{\alpha
-1}T(y)=y^{\alpha }+\frac{c}{\alpha +1}y^{2\alpha }+o(y^{2\alpha })$ is
bounded.

Suppose $y^{\alpha -1}T(y)\leq K,$ then 
\begin{eqnarray*}
&&\frac{2C}{|D_{y}T|^{2}}\frac{y^{\alpha -1}T(y)}{a+bT(y)}+\frac{T(y)}{%
y|D_{y}T|}\frac{a+by}{a+bT(y)} \\
&\leq &\frac{2C}{|D_{y}T|^{2}}\frac{K}{a+bT(y)}+\frac{T(y)}{y|D_{y}T|}\frac{%
a+by}{a+bT(y)}.
\end{eqnarray*}

Now let us find a sufficient condition for which the two summands are less
or equal than $\frac{1}{2}$ and $\frac{1}{2c_{T}}$ (which is $\leq \frac{1}{2%
}$). We will obtain the conditions $\ref{M2},$ $\ref{M3},$ $\ref{M4}$.

For the first summand%
\begin{eqnarray*}
\frac{2C}{|D_{y}T|^{2}}\frac{K}{a+bT(y)} &\leq &\frac{1}{2} \\
\frac{4CK}{|D_{y}T|^{2}} &\leq &a+bT(y)
\end{eqnarray*}%
hence $\frac{4CK}{|D_{y}T|^{2}}\leq a$ is a sufficient condition. For the
second summand%
\begin{eqnarray*}
\frac{T(y)}{y|D_{y}T|}\frac{a+by}{a+bT(y)} &\leq &\frac{1}{2c_{T}} \\
2c_{T}T(y)a+2c_{T}T(y)by &\leq &y|D_{y}T|a+y|D_{y}T|bT(y) \\
2c_{T}T(y)a-y|D_{y}T|a &\leq &y|D_{y}T|bT(y)-2c_{T}T(y)by \\
a_{s}(2c_{T}T(y)-y|D_{y}T|) &\leq &yb(|D_{y}T|T(y)-2c_{T}T(y)) \\
\frac{2c_{T}T(y)-y|D_{y}T|}{(|D_{y}T|-2c_{T})T(y)y} &\leq &\frac{b}{a}.
\end{eqnarray*}
\end{proof}

By Remark \ \ref{cone} we have the following

\begin{proposition}
\label{norrm}Let $T\in M(\alpha ,c,C,C_{3},d),$ suppose $T$ satisfies the
assumptions of Lemma \ref{Mdue}. Then the unique a.c.i.m $h$ is such that%
\begin{equation*}
||h||_{\alpha }\leq \max (A_{\ast },A_{\ast }(a_{T}+b_{T})).
\end{equation*}
\end{proposition}

\begin{proof}
The condition $||h||_{\alpha }\leq A_{\ast }$ follows from $h(x)\leq A_{\ast
}x^{-\alpha }$ (see Remark \ref{cone}). By Lemma \ref{Mdue} we have 
\begin{equation*}
|h^{\prime }(x)|\leq \frac{a_{T}+b_{T}x}{x}h(x)\leq (a_{T}+b_{T}x)A_{\ast
}x^{-\alpha -1}.
\end{equation*}
\end{proof}

\subsubsection{Mixed norm for the difference of the operators for
deterministic perturbations (item 4 of Theorem \protect\ref{gen})}

In this section we see a family of perturbations of maps with an indifferent
fixed point, for which the difference $||(L_{\delta }-L)f||_{1}$ can be
estimated for each $f$ such that $||f||_{\alpha }\leq 1$, allowing the
application of Proposition \ref{pto}. These are examples of perturbations
for which we can have all the required estimations and apply our main
results. Essentially we allow perturbations wich are small in a weighted $%
L^{\infty }$ norm for the map and in the $L^{\infty }$ norm for its
derivative.

\begin{proposition}
Let us consider two maps $T_{0}$, $T_{1}\in $ $M(\alpha ,c,C,d)$, let $%
L_{0},L_{1}$ be their transfer operators. Let $T_{0,1}^{-1}$, $T_{0,2}^{-1}$%
, $T_{1,1}^{-1}$, $T_{1,2}^{-1}$ be the inverse branches of the two maps.\
Suppose that $T_{1}$ is a small perturbation of $T_{0}$ in the following
sense:

\begin{enumerate}
\item[N1] for $i\in \{1,2\},x\in \lbrack 0,1]~,~x^{-\alpha
-1}|T_{0,i}^{-1}(x)-T_{1,i}^{-1}(x)|\leq \epsilon ;$

\item[N2] for all $x\in \lbrack 0,1]~,~||T_{0}^{\prime }(x)-T_{1}^{\prime
}(x)||_{\infty }\leq \epsilon ;$
\end{enumerate}

then there is $C_{6}\geq 0$ not depending on $T_{1}$ such that%
\begin{equation*}
||(L_{0}-L_{1})f||_{1}\leq C_{6}~\epsilon ||f||_{\alpha }.
\end{equation*}
\end{proposition}

\begin{proof}
Writing the transfer operators explicitly:%
\begin{equation*}
L_{0}f(x)=\frac{f(T_{0,1}^{-1}(x))}{T_{0}^{\prime }(T_{0,1}^{-1}(x))}+\frac{%
f(T_{0,2}^{-1}(x))}{T_{0}^{\prime }(T_{0,2}^{-1}(x))},L_{1}f(x)=\frac{%
f(T_{1,1}^{-1}(x))}{T_{1}^{\prime }(T_{1,1}^{-1}(x))}+\frac{%
f(T_{1,2}^{-1}(x))}{T_{1}^{\prime }(T_{1,2}^{-1}(x))}
\end{equation*}%
\begin{eqnarray*}
||L_{0}f(x)-L_{1}f(x)||_{1} &\leq &||\frac{f(T_{0,1}^{-1}(x))}{T_{0}^{\prime
}(T_{0,1}^{-1}(x))}+\frac{f(T_{0,2}^{-1}(x))}{T_{0}^{\prime
}(T_{0,2}^{-1}(x))}-\frac{f(T_{1,1}^{-1}(x))}{T_{1}^{\prime
}(T_{1,1}^{-1}(x))}-\frac{f(T_{1,2}^{-1}(x))}{T_{1}^{\prime
}(T_{1,2}^{-1}(x))}||_{1} \\
&\leq &||\frac{f(T_{0,1}^{-1}(x))}{T_{0}^{\prime }(T_{0,1}^{-1}(x))}+\frac{%
f(T_{0,2}^{-1}(x))}{T_{0}^{\prime }(T_{0,2}^{-1}(x))}-\frac{%
f(T_{1,1}^{-1}(x))}{T_{0}^{\prime }(T_{0,1}^{-1}(x))}-\frac{%
f(T_{1,2}^{-1}(x))}{T_{0}^{\prime }(T_{0,2}^{-1}(x))}||_{1} \\
&&+||\frac{f(T_{1,1}^{-1}(x))}{T_{0}^{\prime }(T_{0,1}^{-1}(x))}+\frac{%
f(T_{1,2}^{-1}(x))}{T_{0}^{\prime }(T_{0,2}^{-1}(x))}-\frac{%
f(T_{1,1}^{-1}(x))}{T_{1}^{\prime }(T_{1,1}^{-1}(x))}-\frac{%
f(T_{1,2}^{-1}(x))}{T_{1}^{\prime }(T_{1,2}^{-1}(x))}||_{1} \\
&\leq &||\frac{f(T_{0,1}^{-1}(x))}{T_{0}^{\prime }(T_{0,1}^{-1}(x))}-\frac{%
f(T_{1,1}^{-1}(x))}{T_{0}^{\prime }(T_{0,1}^{-1}(x))}||_{1}+||\frac{%
f(T_{0,2}^{-1}(x))}{T_{0}^{\prime }(T_{0,2}^{-1}(x))}-\frac{%
f(T_{1,2}^{-1}(x))}{T_{0}^{\prime }(T_{0,2}^{-1}(x))}||_{1} \\
&&+||\frac{f(T_{1,1}^{-1}(x))}{T_{0}^{\prime }(T_{0,1}^{-1}(x))}-\frac{%
f(T_{1,1}^{-1}(x))}{T_{1}^{\prime }(T_{1,1}^{-1}(x))}||_{1}+||\frac{%
f(T_{1,2}^{-1}(x))}{T_{0}^{\prime }(T_{0,2}^{-1}(x))}-\frac{%
f(T_{1,2}^{-1}(x))}{T_{1}^{\prime }(T_{1,2}^{-1}(x))}||_{1} \\
&=&I+II+III+IV
\end{eqnarray*}

Now

\begin{eqnarray*}
I &=&||\frac{f(T_{0,1}^{-1}(x))-f(T_{1,1}^{-1}(x))}{T^{\prime
}(T_{0,1}^{-1}(x))}||\leq ||f(T_{0,1}^{-1}(x))-f(T_{1,1}^{-1}(x))|| \\
&\leq &\int_{0}^{1}|f(T_{0,1}^{-1}(x))-f(T_{1,1}^{-1}(x))|~dx \\
&\leq &\int_{0}^{1}||f||_{\beta }\xi ^{-\beta
-1}|T_{0,1}^{-1}(x)-T_{1,1}^{-1}(x)|~dx
\end{eqnarray*}

by Lagrange theorem, where 
\begin{equation*}
\min (T_{0,1}^{-1}(x),T_{1,1}^{-1}(x))\leq \xi \leq \max
(T_{0,1}^{-1}(x),T_{1,1}^{-1}(x)).
\end{equation*}%
But since $T_{i}\in M(\alpha ,c,C,d)$%
\begin{equation*}
C^{-1}x\leq \min (T_{0,1}^{-1}(x),T_{1,1}^{-1}(x))\leq \max
(T_{0,1}^{-1}(x),T_{1,1}^{-1}(x))\leq Cx
\end{equation*}%
hence%
\begin{eqnarray*}
\int_{0}^{1}|f(T_{0,1}^{-1}(x))-f(T_{1,1}^{-1}(x))|~dx &\leq
&\int_{0}^{1}||f||_{\beta }C^{\beta +1}x^{-\beta
-1}|T_{0,1}^{-1}(x)-T_{1,1}^{-1}(x)|~dx \\
&\leq &C^{\beta +1}\epsilon ||f||_{\beta }.
\end{eqnarray*}

For the summand $II$ the estimation is easier, since $\min
(T_{0,2}^{-1}(x),T_{1,2}^{-1}(x))\geq d\geq 0.$%
\begin{eqnarray*}
||\frac{f(T_{0,2}^{-1}(x))}{T_{0}^{\prime }(T_{0,2}^{-1}(x))}-\frac{%
f(T_{1,2}^{-1}(x))}{T_{0}^{\prime }(T_{0,2}^{-1}(x))}|| &\leq &\int
|f(T_{0,2}^{-1}(x))-f(T_{1,2}^{-1}(x))|dx \\
&\leq &\int_{0}^{1}||f||_{\beta }d^{-\beta
-1}|T_{0,1}^{-1}(x)-T_{1,1}^{-1}(x)|~dx \\
&\leq &d^{-\beta -1}\epsilon ||f||_{\beta }.
\end{eqnarray*}

For the summand $III:$ 
\begin{eqnarray}
III &=&||\frac{f(T_{1,1}^{-1}(x))}{T_{0}^{\prime }(T_{0,1}^{-1}(x))}-\frac{%
f(T_{1,1}^{-1}(x))}{T_{1}^{\prime }(T_{1,1}^{-1}(x))}||_{1}  \label{summm} \\
&\leq &||\frac{f(T_{1,1}^{-1}(x))}{T_{0}^{\prime }(T_{0,1}^{-1}(x))}-\frac{%
f(T_{1,1}^{-1}(x))}{T_{0}^{\prime }(T_{1,1}^{-1}(x))}||_{1}+||\frac{%
f(T_{1,1}^{-1}(x))}{T_{0}^{\prime }(T_{1,1}^{-1}(x))}-\frac{%
f(T_{1,1}^{-1}(x))}{T_{1}^{\prime }(T_{1,1}^{-1}(x))}||_{1}  \notag
\end{eqnarray}

The last two summands can be estimated as follows%
\begin{eqnarray*}
||\frac{f(T_{1,1}^{-1}(x))}{T_{0}^{\prime }(T_{0,1}^{-1}(x))}-\frac{%
f(T_{1,1}^{-1}(x))}{T_{0}^{\prime }(T_{1,1}^{-1}(x))}||_{1} &\leq &\int
|f(T_{1,1}^{-1}(x))|~dx~||\frac{1}{T_{0}^{\prime }(T_{0,1}^{-1}(x))}-\frac{1%
}{T_{0}^{\prime }(T_{1,1}^{-1}(x))}||_{\infty } \\
&\leq &||f||_{\alpha }\int (T_{1,1}^{-1}(x))^{-\alpha }~dx||\frac{1}{%
T_{0}^{\prime }(T_{0,1}^{-1}(x))}-\frac{1}{T_{0}^{\prime }(T_{1,1}^{-1}(x))}%
||_{\infty } \\
&\leq &Const_{1}.||f||_{\alpha }||\frac{-T_{0}^{\prime \prime }}{%
(T_{0}^{\prime })^{2}}(\xi )~(T_{0,1}^{-1}(x)-T_{1,1}^{-1}(x))||_{\infty }
\end{eqnarray*}

for some $\xi $ such that 
\begin{equation*}
C^{-1}x\leq \min (T_{0,1}^{-1}(x),T_{1,1}^{-1}(x))\leq \xi \leq \max
(T_{0,1}^{-1}(x),T_{1,1}^{-1}(x))\leq Cx
\end{equation*}%
\ as before. Since $|T_{0,i}^{-1}(x)-T_{1,i}^{-1}(x)|\leq \epsilon x^{\alpha
+1}$ and $|T_{0}^{\prime \prime }(x)|\leq Cx^{\alpha -1},$ then 
\begin{equation*}
\sup |\frac{-T_{0}^{\prime \prime }}{(T_{0}^{\prime })^{2}}(\xi
)(T_{0,1}^{-1}(x)-T_{1,1}^{-1}(x))|\leq C^{2}\epsilon
\end{equation*}%
hence%
\begin{equation*}
||\frac{f(T_{1,1}^{-1}(x))}{T_{0}^{\prime }(T_{0,1}^{-1}(x))}-\frac{%
f(T_{1,1}^{-1}(x))}{T_{0}^{\prime }(T_{1,1}^{-1}(x))}||_{1}\leq
Const_{2}.~\epsilon ||f||_{\alpha }.
\end{equation*}%
For the other summand in (\ref{summm})%
\begin{eqnarray*}
||\frac{f(T_{1,1}^{-1}(x))}{T_{0}^{\prime }(T_{1,1}^{-1}(x))}-\frac{%
f(T_{1,1}^{-1}(x))}{T_{1}^{\prime }(T_{1,1}^{-1}(x))}||_{1} &\leq & \\
Const.~||f||_{\alpha }||T_{0}^{\prime }-T_{1}^{\prime }||_{\infty } &\leq
&Const.~\epsilon ||f||_{\alpha }.
\end{eqnarray*}

$IV$ can be treated similarly.
\end{proof}

By what is proved above we can state the following quantitative statistical
stability under deterministic perturbations result:

\begin{theorem}
\label{pto copy(1)}Let us consider a family of maps $T_{s}$ , $s\in \lbrack
0,1)$ , such that $T_{0}\in M_{0}(\alpha ,c,C,d)$ and $T_{s}\in M(\alpha
,c,C,C_{3,}d)$. Suppose that each $T_{s}$ satisfy the assumptions ($\ref{M2}$%
), ($\ref{M3}$), ($\ref{M4}$) of Lemma \ref{Mdue} uniformly (whit uniform
constants for each $s$ ). Let us suppose that also assumptions $N1,N2$ are
satisfied, i.e:

\begin{itemize}
\item for each $i\in \{0,1\},$ $s,x\in \lbrack 0,1]~,~x^{-\alpha
-1}|T_{0,i}^{-1}(x)-T_{s,i}^{-1}(x)|\leq \epsilon _{s};$

\item for all $s,x\in \lbrack 0,1]~,~||T_{0}^{\prime }(x)-T_{s}^{\prime
}(x)||_{\infty }\leq \epsilon _{s}.$
\end{itemize}

Let $f_{s}$ the absolutely continuous invariant measures of $T_{s}$, then
for each $0<\gamma <\frac{1}{\alpha }-1$ 
\begin{equation*}
||f_{s}-f_{0}||_{1}\leq O(\epsilon _{s}^{1-\frac{1}{\frac{\gamma }{2}%
(1-\alpha )+1}}).
\end{equation*}
\end{theorem}

\begin{remark}
We remark that this statement gives Holder continuity if $1-\frac{1}{\frac{%
\gamma }{2}(1-\alpha )+1}>0$. This is always verified for $0\leq \alpha <1.$
\end{remark}

\end{document}